\documentclass{amsart} 

\newtheorem{thm}{Theorem}[section]

\newtheorem{prop}[thm]{Proposition}
\newtheorem{lem}[thm]{Lemma}
\newtheorem{defn}[thm]{Definition}

\theoremstyle{definition}\newtheorem{question}[thm]{Question}

\theoremstyle{remark}\newtheorem{remark}[thm]{Remark}

\theoremstyle{definition}\newtheorem{example}[thm]{Example}

\def\R{{\mathbb R}}

\def\N{{\mathbb N}}

\def \air{{\vskip 12pt\noindent}\par}

\def \ib{\begin{enumerate}}
\def \ie{\end{enumerate}}

\def \det{{\;\mbox{\rm det}}}

\def \Id {{\rm Id}}

\def\ib {\begin{enumerate}}
\def\ie {\end{enumerate}}

\def\adots{\mathinner{\mkern2mu\raise 1pt\hbox{.}\mkern 3mu\raise
4pt\hbox{.}\mkern1mu\raise 7pt\hbox{{.}}}}

\def \diag {{\rm diag}}

\title{Piecewise Certificates of Positivity for matrix polynomials}
\author{Ronan Quarez}

\address{IRMAR (CNRS, URA 305), Universit\'e de Rennes 1, Campus de Beaulieu\\ 35042 Rennes Cedex, France} 
\email{e-mail : ronan.quarez@univ-rennes1.fr}
\date{\today} 


\begin{document}

\begin{abstract}
We show that any symmetric positive definite homogeneous matrix polynomial $M\in\R[x_1,\ldots,x_n]^{m\times m}$ admits a piecewise semi-certificate, i.e. a collection of identites $M(x)=\sum_jf_{i,j}(x)U_{i,j}(x)^TU_{i,j}(x)$ where $U_{i,j}(x)$ is a matrix polynomial and $f_{i,j}(x)$ is a non negative polynomial on a semi-algebraic subset $S_i$, where $\R^n=\cup_{i=1}^r S_i$. This result generalizes to the setting of biforms.\par
Some examples of certificates are given and among others, we study a variation around the Choi counterexample of a positive semi-definite biquadratic form which is not a sum of squares.
As a byproduct we give a representation of the famous non negative sum of squares polynomial $x^4z^2+z^4y^2+y^4x^2-3\,x^2y^2z^2$ as the determinant of a positive semi-definite quadratic matrix polynomial.      
\end{abstract}

\maketitle

\section{Introduction}
The Hilbert 17-th problem asks if a non negative polynomial $f(x)\in\R[x_1,\ldots,x_n]$ is a sum of squares. The answer given by Artin provides an identity with denominators (or rational functions), namely there are two sums of squares of polynomials $q(x)$ and $p(x)$ such that $q(x)f(x)=p(x)$.\par
Such an identity is a called Positivestellensatz and can be seen as a {\it certificate} which algebraically proves the positivity of the polynomial $f(x)$. For a  general polynomial, the denominator $q(x)$ is necessary. 
\par One may also consider a "relative" version, namely when $f(x)$ is non negative on a basic semi-algebraic subset $S=\{x\in\R^n\mid f_1(x)\geq 0,\ldots,f_r(x)\geq 0\}$. This search of certificates received a lot of contributions especially for those semi-algebraic subsets $S$ where the Positivestellensatz is denominator free. Essentially it happens to be possible when $f(x)$ is positive and when the description of $S$ satisfies some archimedean property (which can be seen as a ``strong" compactedness assumption). Among all the available references, we mention the seminal works of Schm\"udgen \cite{Sg} and Putinar \cite{Pu}.\par
Another natural and often difficult question is, when it is possible to obtain a certificate of positivity, to study its complexity, i.e. the number of squares that are needed.
\air

In this paper, we are mainly interested in matrix polynomials, namely polynomials whose coefficients are matrices. Let $x=(x_1,\ldots,x_n)$ and $\R[x]^{m\times m}$ be the set of all square matrix polynomials of size $m$ with entries in $\R[x]$.
In that context, by a square we mean a expression of the form $U(x)^TU(x)$ where $U(x)\in \R[x]^{m\times m}$, sometimes we will also use the terminology of {\it hermitian} square.\par
Given a symmetric matrix polynomial $M(x)$, the natural question is the following :\par Assume that $M(x)$ is positive semi-definite ({\it psd} in short) for all $x\in\R^n$. Is $M(x)$ a sum of hermitian squares ? Or in otherwords, do we have a positivity certificate for $M(x)$ ?\par
It is well known that, when $n=1$, then $M(x)$ is a sum of two squares (see for instance \cite{Dj} for a proof), i.e. $M(x)=U(x)^TU(x)+V(x)^TV(x)$ where $U(x),V(x)\in \R[x]^{m\times m}$.\par
There is also another case where we have a positive answer : the case of biquadratic forms in dimension $2$. Namely, if $m=2$ and if the entries are homogeneous polynomials of degree $2$, then $M(x)$ is also a sum of squares.\par
In fact this last case can be seen as a particular case of the homogenized version of the previous one (see \cite{CLR1} for this fact and a proof). \par
Unfortunaltely these are essentially the only cases where such a certificate exists for a matrix polynomial. For instance, it becames false even for biquadratic forms in $3$ variables and dimension $3$ (confer the Choi counterexample in \cite{C}).\par
\par We may also look at the relative situation : if $M(x)$ is positive definite for all $x$ in a semi-algebaric subset $S\subset\R^n$ whose description satisfies an archimedean property, then we get  a similar certificate as in the polynomial case. This has been proved by Hol and Scherer in \cite{HS} ; see also \cite{KS} for a different and more algebraic proof.\air

In this paper, we are searching for certificates of positivity {\it without denominators} for matrix polynomials. We first study certificates which are sums of squares of matrix polynomial weighted by a psd (scalar) polynomial.  More precisely, given a psd matrix polynomial $M(x)$, we wonder if $M(x)$ can be written as a finite sum
\begin{equation}\label{without-denom}
M(x)=\sum_if_i(x)U_i(x)^TU_i(x)
\end{equation}
 where $U_i(x)\in \R[x]^{m\times m}$ and $f_i(x)$ is non negative polynomial on $\R^n$.\par We call it a {\it semi-certificate} since, roughly speaking, we have a certificate which is a hermitian sum of squares weighted by some coefficients which are non negative polynomials. 
\par Of course such a certificate may gives some others which will be totally algebraic by using usual certificates for the psd polynomials $f_i(x)$. Beware that using denominators at this point could break any interest since we may readily get one such certificate with denominators simply by using the Gauss algorithm (indeed, it suffices to inverte the $m-1$ first principal minors of the given matrix polynomial). Although, whene the matrix polynomial is positive definite, a relative version of semi-certificate with respect to a semi-algebraic subset $S$ whose desciption is archimedean would produce a usual denominator free algebraic certificate for the matrix polynomial on $S$, using for instance Schm\"udgen or Putinar certificates for all the $f_{i,j}$'s.\par
Unfortunately such a semi-certificate as in (\ref{without-denom}) does not exist in general, even for a positive definite matrix polynomial, as it is illustrated in section 3. \par 
Nevertheless, noticing that one may obtain such certificate locally, we introduce the notion of {\it piecewise} semi-certificate, i.e. a collection of semi-certificates as in (\ref{without-denom}) with respect to a semi-algebraic covering of $\R^n$.\par We show in section 4, that such a certificate exists for any {\it positive definite} homogeneous matrix polynomial. We show also that we may translate the notion of semi-certificate and our related results to the case of biforms.\par
Some examples and counterexamples are given and for instance, in section 5, we study some variations around the Choi counterexample given in \cite{C}.  \par
As a byproduct we give a representation of the famous psd non sum of squares polynomial $x^4z^2+z^4y^2+y^4x^2-3x^2y^2z^2$ as the determinant of a psd quadratic matrix polynomial.   
\air
Of course, a lot a work have to be done to better understand the framework of semi-certificate for psd non definite matrix polynomials. Moreover, concerning complexity, one may be interested in estimating the lenght of the hermitian sum of squares representation, and also in the number of pieces of the semi-certificate. Since we use a compactedness argument, no bound on the number of pieces of the semi-algebraic covering can be derived by our method.

\section{Notations}
\subsection{Matrix polynomials and biforms}
In all the paper, we consider a set of variables $x=(x_1,\ldots,x_n)$ and the polynomial ring $\R[x]$ whose associated field of rational functions is $\R(x)$.\par

A matrix polynomial $A(x)$ is just a matrix whose entries are polynomials. For instance, for a square matrix polynomial of size $m$ in the variables $x$, we write $A(x)\in\R[x]^{m\times m}$.\par 

A {\it form} of type $(n,d)$ will denote an homogeneous polynomial of degree $d$ in $n$ variables. A {\it biform}  of type $(n_1,d_1;n_2,d_2)$ will denote a homogeneous polynomial in $n_1+n_2$ variables which is of degree $d_1$ with respect to a set of $n_1$ variables and of degree $d_2$ with respect to the other set of $d_2$ variables. \par
To a matrix polynomial $A(x)=(a_{i,j}(x))\in\R[x]^{m\times m}$ whose entries are all forms of type $(n,d)$, we may canonically associate a biform of type $(n_1,d_1;m,2)$ : $$f_A(x,y)=\sum_{1\leq i,j\leq m}a_{i,j}(x)y_iy_j.$$

We say that a biform of type $(n_1,d_1;n_2,d_2)$ is 
{\it positive definite} if $f(x,y)>0$ for all $(x,y)\in S^{n_1-1}\times S^{n_2-1}$, where $S^{k-1}$ denotes the $k_1$ dimentional unit sphere in $\R^k$. Equivalentely, we may alternatively consider the product of projective spaces instead of unit spheres. \par
If $A(x)$ is a symmetric matrix polynomial in $\R[x]^{m\times m}$ whose entries are all forms of type $(n,d)$, we say that $A(x)$ is positive definite if the associated biform $f_A(x,y)$ is positive definite. We may also note that $A(x)$ is psd if and only if $f_A(x,y)$ is psd on the product of unit spheres.

\subsection{Sum of squares and Hilbert 17th problem}
Note that sums of hermitian squares when dealing with matrix polynomials corresponds to usual sum of squares in the terminology of biforms. \air

Let $A(x)$ be a symmetric matrix polynomial in $\R[x]^{m\times m}$ which is positive semi-definite for all substitution of $x\in\R^n$. Then, the Gauss reduction algorithm gives a solution with denominators to the Hilbert $17$-th Problem for a generic matrix polynomial : there exist $r\in \N$, $U_1(x),\ldots,U_r(x)\in\R[x]^{m\times m}$ and $f_1(x)\ldots, f_r(x),f(x)\in\R[x]$ psd polynomials such that 
\begin{equation}\label{sos_matrices}f(x)A(x)=\sum_{i=1}^rf_i(x)U_i(x)^TU_i(x)\end{equation}
where $f(x)$ can be choosen as the product of the $m-1$ first minors of $A(x)$ when they do not identically vanish. Moreover, by multiplying the identity (\ref{sos_matrices}) by the denomiators appearing in a positivity certificate for the $f_i$'s, we may assume that the polynomials $f_i$ and $f$ are sum of squares.\air

In our paper, we are mainly interested in positivity certificates without denominators, namely when we may choose $f\equiv 1$.  Hence, we first look for matrix polynomials which can be written as in (\ref{without-denom}) :
$$A(x)=\sum_{i=1}^rf_i(x)U_i(x)^TU_i(x)$$
where the $f_i$'s are non-negative polynomials (we cannot ask for them to be sum of squares since denominators are already needed for polynomials).

\air
One may want to measure the complexity of the certificate given in (\ref{without-denom}), simply by counting the  number of squares : $r$. But, one may also prefer another way of counting.\par

Noticing that 
$U^TU=U^T(E_1+\ldots+E_m)
U=(E_1U)^T(E_1U)+\ldots+(E_mU)^T(E_mU)$
where $E_i$ is the diagonal matrix whose only non null entry is $1$ at the $i$-th position onto the diagonial, we may prefer to count ``rank-one" sums of squares, i.e. when $U(x)$ has rank one in $\R(x)^{m\times m}$. The advantage of that count is that it coincides with the number of needed squares when we view our matrix polynomial as a biform.

\section{Semi-certificates}
The following examples shows that it is not possible in general to hope for (\ref{without-denom}).

\subsection{Some examples}
\begin{example}\label{first-example}
Let : $$M=\left(\begin{array}{cc}
1+x^2&xy\\
xy&x^2+y^4
\end{array}\right)$$
Of course $M$ is psd for all $(x,y)\in\R^2$.\air

To show that $M$ cannot be written as in (\ref{without-denom}), we may proceed by applying to the biform $f_M$ which is canonically associated to $M$) the technology of cages and Gram matrices developped in \cite{CLR2}. But we prefer to produce here an elementary and self-contained argument.\par

Assume that $M=\sum_{i=1}^rf_iU_i^TU_i$, where $U_i\in(\R[x,y])^{2\times 2}$ and $f_i\geq 0$ on all $\R^2$. 

By decomposing into rank one sum of squares, we may assume that  $$U_i^TU_i=\left(\begin{array}{cc}
a_i^2&a_ib_i\\
a_ib_i&b_i^2\end{array}\right), \quad {\rm where}\; a_i,b_i\in\R[x,y].$$
  Identifying the $(1,1)$-entry, we get $\sum_{i=1}^rf_ia_i^2=1+x^2$. Thus, $f_i$ does not depend on $y$ : $f_i\in\R[x]$. Since $f_i\geq 0$ for all $x\in\R$, we deduce that $f_i$ is a sum of two squares in $\R[x]$. Thus, if we increase the integer $r$ and if we change the $U_i$'s, we may assume that $f_i=1$. Identifying the entries, we get 
$$\left\{\begin{array}{rcl}
\sum_{i=1}^ra_i^2&=&1+x^2\\
\sum_{i=1}^rb_i^2&=&x^2+y^4\\
\sum_{i=1}^ra_ib_i&=&xy
\end{array}\right.$$
By a simple degree consideration, the polynomial $b_i$ has necessarily the form $$b_i=\alpha_ix+\beta_iy^2,\; {\rm with}\; \alpha_i,\beta_i\in\R.$$
Likewise, the polynomial $a_i$ has necessarily the form $$a_i=\gamma_i+\delta_i x\; {\rm with}\; \gamma_i,\delta_i\in\R.$$ It follows a contradiction with the identity $\sum_{i=1}^ra_ib_i=xy$.
\end{example}

The previous pathology is not due to the fact that the matrix $M$ is only psd, as shown by considering the following variation :

\begin{example}\label{second-example}
Let 
$$N=\left(\begin{array}{cc}
1+x^2+\epsilon (x^4+y^4)&xy\\
xy&\epsilon (1+x^4)+x^2+y^4
\end{array}\right)$$
where $\epsilon>0$.
This is a positive definite matrix polynomial on all $\R^2$ (and even at infinity, i.e. the associated homogeneized matrix polynomial remains positive definite).\par
Assume that $N$ can be written as in (\ref{without-denom}) :
$$ \begin{array}{ccl}
N&=&\displaystyle\sum_{i=1}^rf_i(x,y)\left(\begin{array}{cc}
A_i^2&A_iB_i\\
A_iB_i&B_i^2\end{array}\right)
\\
&+&
\displaystyle\sum_{j=1}^r{g_j}(x,y)\left(\begin{array}{cc}
U_j^2&U_jV_j\\
U_jV_j&V_j^2\end{array}\right)
\\ &+&
\sum_{k=1}^r\left(\begin{array}{cc}
R_k^2&R_kS_k\\
R_kS_k&S_k^2\end{array}\right)
\end{array}$$ 

where the $f_i$'s are polynomials of degree at most $4$, the $g_j,R_k,S_k$'s are polynomials of degree at most $2$, the $U_j,V_j$'s are polynomials of degree at most $1$ and the $A_i,B_i$'s are constant. Moreover, the polynomials $f_i,g_j$ are assumed to be non negative on $\R^2$. Since their degrees are at most $4$ and the number of variables is $2$, we know that they are sum of squares of polynomials. Hence, we may assume that 
\begin{equation}\label{second-example-non-sos}
N=\sum_{i=1}^r\left(\begin{array}{cc}
R_i^2&R_iS_i\\
R_iS_i&S_i^2\end{array}\right)
\end{equation} 

where 
$$\left\{\begin{array}{l}
R_i=a_i+b_ix+c_ix^2+d_iy+e_iy^2+f_ixy,\\
S_i=\alpha_i+\beta_ix+\gamma_ix^2+\delta_iy+\nu_iy^2+\mu_ixy,\\
(a_i,b_i,c_i,d_i,e_i,f_i,\alpha_i,\beta_i,\gamma_i,\delta_i,\nu_i,\mu_i)\in\R^{12}.
\end{array}\right.$$

Let us consider the following vectors in the usual euclidien space $\R^r$ : 
$$\begin{array}{l}{\bar a}=(a_i),{\bar b}=(b_i),{\bar c}=(c_i),{\bar d}=(d_i),{\bar e}=(e_i),{\bar f}=(f_i),\\
{\bar \alpha}=(\alpha_i),{\bar \beta}=(\beta_i),{\bar \gamma}=(\gamma_i),{\bar \delta}=(\delta_i),{\bar \nu}=(\nu_i),{\bar \mu}=(\mu_i).
\end{array}$$ 

Identifying the non diagonal entries in (\ref{second-example-non-sos}), we get  
\begin{equation}
\label{contradiction-second-example}{\bar b}\cdot{\bar\delta}+\bar\beta\cdot \bar d+\bar a\cdot\bar\mu+\bar f\cdot\bar\alpha=1
\end{equation}

where $\cdot$ denotes the usual inner product on $\R^r$.\par

By identifying the diagonal entries in (\ref{second-example-non-sos}) and by the Cauchy-Schwarz inequality, we get

$$\left\{\begin{array}{rclclcl}
\bar b^2&=&1-2\,\bar a\cdot\bar c&\leq& 1+2\sqrt{\bar a^2}\sqrt{\bar c^2}&\leq& 1+2\sqrt\epsilon\\
\bar\delta^2&=&-2\,\bar\alpha\cdot\bar\nu&\leq& 2 \sqrt{\bar\alpha^2}\sqrt{\bar\nu^2}&\leq& 2\sqrt\epsilon\\
\bar \beta^2&=&1-2\,\bar \alpha\cdot\bar \gamma&\leq& 1+2\epsilon&&\\
\bar d^2&=&-2\,\bar a\cdot\bar c&\leq& 2\sqrt\epsilon&&\\
\bar a^2&=&1&&&&\\
\bar \mu^2&=&-2\,\bar \gamma\cdot\bar\nu&\leq& 2\sqrt\epsilon&&\\
\bar f^2&=&-2\,\bar c\cdot \bar e&\leq& 2\sqrt{\bar c^2}\sqrt{\bar e^2}&\leq& 2\epsilon\\
\bar\alpha^2&=&\epsilon&&&&
\end{array}\right.$$

Again, by the Cauchy-Schwarz inequality, we have :
$$\left\{\begin{array}{rcl}
{\bar b}\cdot{\bar\delta}&\leq& \sqrt{2(1+2\sqrt\epsilon)\sqrt{\epsilon}}\\
{\bar \beta}\cdot{\bar d}&\leq& \sqrt{2(1+2\epsilon)\sqrt{\epsilon}}\\
\bar a\cdot\bar \mu&\leq& \sqrt{2\sqrt{\epsilon}}\\
\bar f\cdot \bar \alpha&\leq& \sqrt{2}\epsilon
\end{array}\right.$$

Thus, if we take $\epsilon$ small enough, then we get a contradiction to (\ref{contradiction-second-example}). Namely, $N$ does not admit a certificate as in (\ref{without-denom}). 
\end{example}

\subsection{Semi-algebraic covering}
Thus, even for positive definite matrix polynomials, a certificate (\ref{without-denom}) does not exist in general. Although, one may remark that such a certificate exists {\it locally} in our examples. \air

Look at example \ref{first-example}. When $|x | \geq |y |$, then consider the identity
$$\left(\begin{array}{cc}
1+x^2&xy\\
xy&x^2+y^4
\end{array}\right)=\left(\begin{array}{cc}
x^2&xy\\
xy&y^2
\end{array}\right)+\left(\begin{array}{cc}
1&0\\
0&0
\end{array}\right)+(x^2-y^2+y^4)\left(\begin{array}{cc}
0&0\\
0&1
\end{array}\right)$$

whereas when $|y | \geq |x |$, then consider
$$\left(\begin{array}{cc}
1+x^2&xy\\
xy&x^2+y^4
\end{array}\right)=\left(\begin{array}{cc}
1&xy\\
xy&x^2y^2
\end{array}\right)+x^2 \left(\begin{array}{cc}
1&0\\
0&1
\end{array}\right)+(y^4-x^2y^2)\left(\begin{array}{cc}
0&0\\
0&1
\end{array}\right)$$

Hence we are going to change our definition by considering a semi-algebraic covering of our space.\par From now on, we will focus on {\it forms}, namely {\it homogeneous} polynomials.

\begin{defn}
Let $f(x,y)$ be a biform of type $(n,d_1;m,d_2)$. A {\it semi-certificate} of positivity is the data of
a semi-algebraic covering $\R^{n}=S_1\cup\ldots\cup S_r$ such that 
$$\forall x\in S_i,\quad f(x,y)=\sum_{j=1}^{r_i}f_{i,j}(x)(g_{j,i}(x,y))^2$$
where each $f_{i,j}(x)$ is a form of degree $e_{i,j}$ which is non negative on $S_i$ and
$g_{i,j}^2(x,y)$ is a biform of type $(n,d_1-e_{i,j};m,d_2)$.
\end{defn}

Roughly speaking, the semi-certificate is a piecewise identity which is a sum of squares with respect to one set of variables $y$ and only psd with respect to the other set of variables $x$.

\air
For convenience, we may formulate what happens when we are dealing with matrix polynomials. The symmetric psd matrix polynomial $M(x)\in \R[x]^{m\times m}$ admits a semi-certificate if there exists a (finite) semi-algebraic covering $(S_i)_{1\leq i\leq r}$ of $\R^n$ and 
$$(\forall i)(\exists r_i\in\N)(\forall j\in\{1,\ldots,s_i\})(\exists f_{i,j}(x)\in\R[x])(\exists U_{i,j}(x)\in \R[x]^{m\times m})$$
such that
\begin{equation}\label{semi-certif}\left(M(x)=\sum_{j=1}^{r_i}f_{i,j}(x)U_{i,j}^T(x)U_{i,j}(x)\right) \;{\rm and}\; \left(\forall x\in S_i,\, f_{i,j}(x)\geq 0\right)\end{equation}

\air
In example \ref{first-example}, we have seen that it is possible to produce semi-certificate of positivity, where the considered semi-algebraic covering of $\R^2$ necessarily have more than one piece.    

\remark{ 
Note that if we restrict ourselves to certificates where the $U_{i,j}(x)$ are constant matrices, then we a strictly smaller class of certificates. For instance, consider the matrix polynomial 
$\left(\begin{array}{cc}
x^2&xy\\
xy&y^2\end{array}\right)$ 
which is psd on all $\R^2$, which is obviously psd 
 on a neighbourhood $V$ of the point $(1,1)$ ( we may proceed likewise the neighbourhood of any given point).\par
 Let us deshomogenize by setting $y=1$, set for simplicity $V=[1,1+\epsilon[$, $\epsilon>0$,  and assume that 
$$\left(\begin{array}{cc}
x^2&x\\
x&1\end{array}\right)=\sum_if_i(x)
\left(\begin{array}{cc}
\alpha_i^2&\alpha_i\\
\alpha_i&1\end{array}\right)+\sum_jg_j(x)\left(\begin{array}{cc}
1&0\\
0&0\end{array}\right)$$  where the $f_i$'s and $g_j$'s are univariate polynomials psd on $V$ and $\alpha_i\in\R$.\par
Let $$\left\{\begin{array}{l}f_i(x)=a_i+b_i(x-1)+c_i(x-1)^2\\
g_j(x)=d_j+e_j(x-1)+f_j(x-1)^2\end{array}\right.$$  Then 
$$
\left\{\begin{array}{ccl}
x^2&=&\sum_i\alpha_i^2f_i(x)+\sum_jg_j(x)\\
x&=&\sum_i\alpha_if_i(x)\\
1&=&\sum_if_i
\end{array}\right.$$\par
Combining these three identities we get 
$$\sum_i(\alpha_i-1)^2f_i(x)+\sum_jg_j(x)=(x-1)^2$$
By assumption, we have $a_i\geq 0$ and $d_j\geq 0$. Then, for all $j$ we have $d_j=0$ and $e_j=0$. Furthermore, if $\alpha_i\not =1$, then $a_i=0$ and also $b_i=0$. But substracting the last two equalities of the previous system yields
$$x-1=\sum_i(\alpha_i-1)f_i(x)$$
a contradiction.
}

\section{Main results} 

There is one case when semi-certificates exist : when the matrix polynomial is positive definite. Remind that it means that it is positive definite on the unit sphere or the associated projective space.

\begin{thm}\label{main_th} Let $A(x)\in(\R[x])^{m\times m}$ be a positive definite marix polynomial. Then, there is a finite semi-algebraic covering of $\R^n=\cup_{i=1}^rS_i$, some forms $f_{i,j}(x)\in\R[x]$, some matrix polynomials $A_{i,j}(x)\in(\R[x])^{m\times m}$ such that,  
$$\forall x\in S_i,\;A(x)=\sum_{j=1}^{r_i}f_{i,j}(x)A^T_{i,j}(x)A_{i,j}(x)$$
with the condition that $f_{i,j}(x)\geq 0$ for all $x\in S_i$.
\end{thm} 
\begin{proof}

Assume that the entries of $A(x)$ are $d$-forms ($d$ is even). \par
We start with a lemma 
\begin{lem}\label{pd_open}
Let $a(x)$ be a positive definite $d$-form and $c(x)$ be a positive semi-definite $d$-form. 
Then, there exists $\epsilon>0$ such that $a(x)-\epsilon c(x)$ remains positive definite.
\end{lem}
\begin{proof}
The $d$-form $\widetilde{c}(x)=c(x)+x_1^d+\ldots+x_n^d$ is positive definite and the subset $S_{\tilde{c}}=\{x\mid \tilde{c}(x)=1\}$ is compact.
Since $a(x)>0$ on $S_{\tilde{c}}$ we have for all $x\in S_{\tilde{c}}$, $a(x)\geq m>0$. Hence,  for all $x$, $$a\left(\frac{x}{\sqrt[d]{\tilde{c}(x)}},\frac{x}{\sqrt[d]{\tilde{c}(x)}}\right)\geq m>0.$$
Then $a(x)-\frac{m}{2} \tilde{c}(x)$ is positive definite, which concludes the proof since $\tilde{c}(x)\geq c(x)$. 
\end{proof}

Let $A(x)=(a_{i,j}(x))_{1\leq i,j\leq m}$. For a given $x_0\in\R^n\setminus\{0\}$, the matrix $A(x)$ is positive definite in a semi-algebraic subset $U_{x_0}$ which is open in the unit sphere $S^{n-1}$ of $\R^n$.\par
Let $U$ be the square matrix of size $m$ whose all entries are equal to $1$. By \ref{pd_open} and up to resizing $U_{x_0}$, we may assume that $B(x)=A(x)-\epsilon (x_1^d+\ldots+x_n^d)U$ remains positive definite in $U_{x_0}$ for $\epsilon$ small enough, and moreover that 
all entries of $B(x_0)$ are non-zero.\par  
Let $B(x)=(b_{i,j}(x))_{1\leq i,j\leq m}$. Assume that $b_{2,1}(x_0)> 0$.
We may write $B(x_0)=b_{2,1}(x_0)C$ where $C\in\R^{m\times m}$ is positive definite. Let $h_{x_0}(x)$ be a $d$-form satisfying $h_{x_0}(x_0)=1$. Then, by \ref{pd_open}, $$b_{2,1}(x)C-\epsilon_1 h_{x_0}(x)\Id_m$$ remains positive definite for some small $\epsilon_1>0$. \par
Now write
$$B(x)=\left(b_{2,1}(x)C-\epsilon_1 h_{x_0}(x)\Id_m\right)+\widetilde{B}(x)$$
Since $\widetilde{B}(x_0)=\epsilon_1 \Id_m$, the matrix $\widetilde{B}(x)$ is positive definite in an open semi-algebraic neighbourhood of $x_0$ which we still denote by  $U_{x_0}$.\par
Then, $\widetilde{B}(x))=(\tilde{b}_{i,j}(x))_{1\leq i,j\leq m}$ is such that $\tilde{b}_{2,1}(x)=0$.\par 
Of course, we proceed likewise when $b_{2,1}(x_0)< 0$, writing $B(x_0)=-b_{2,1}(x_0)C$ where $C\in\R^{m\times m}$ is positive definite. Thus, 
$$B(x)=\left(-b_{2,1}(x)C-\epsilon_1 h_{x_0}(x)\Id_m\right)+\widetilde{B}(x)$$
\air 

Then, we repeat the same process to get rid off all the $b_{i,j}(x)$'s such that $i\not = j$ and reduce, up to resize $U_{x_0}$, to the case where $B(x)$ is diagonal. Namely, there is an open semi-algebraic subset $U_{x_0}\subset S^{n-1}$ such that for all $x\in U_{x_0}$,
\begin{equation}\label{Ux0}
A(x)=\sum_{k}f_k(x)U^T_kU_k
\end{equation}
where $U_k\subset \R^{m\times m}$ is a constant matrix and $f_k(x)$ is a $d$-form which is psd on $U_{x_0}$. 

\air
To conclude the proof, we extract a finite semi-algebraic covering by compactedness of $S^{n-1}$.\par
By homogeneity, we extend the identities (\ref{Ux0}) to all $x\in \R^n$. Moreover, we may manage in order that the polynomials describing the $U_{x_0}$'s are forms, and hence the identity (\ref{Ux0}) is true on $S_{x_0}$, the cone in $\R^n$ with origin $0$ and basis  $U_{x_0}$. Thus, we get a finite covering of $\R^n$.
\end{proof}

Note that the proof gives an {\it open} semi-algebraic covering, and that the $U_k$'s are constant matrices. 

\air
Likewise, we have an analogeous result for biforms : 
\begin{thm} Let $f(x,y)\in\R[x,y]$ be a positive definite biform of type $(n,d_1;m,d_2)$. Then, there is a finite semi-algebraic covering of $\R^n=\cup_{i=1}^rS_i$, some forms $f_{i,j}(x)\in\R[x]$, some biforms $g_{i,j}(x)\in\R[x,y]$ such that,  
$$\forall x\in S_i,\;f(x,y)=\sum_{j=1}^{r_i}f_{i,j}(x)(g_{i,j}(x,y))^2$$
with the condition that $f_{i,j}(x)\geq 0$ for all $x\in S_i$.
\end{thm} 
\begin{proof}
We proceed as in the proof of Theorem \ref{main_th}. We get rid of any monomial appearing in $f$ with some odd power with respect to at least one indeterminate.  
\end{proof}

\section{Semi-certificates on orthant-neighbourhoods}
We say that $V$ is an {\it orthant-neighbourhood} of $x_0$ if it contains the intersection of a neighbourhood of $x_0$ and an orthant $O_{x_0}$ centered at $x_0$ (i.e. a subset defined by an open condition $x=x_0+(X_1,\ldots,X_n)\in O_{x_0}$ if $\epsilon_1X_1>0,\ldots,\epsilon_nX_n>0$ where $(\epsilon_1,\ldots,\epsilon_n)\in\{-1,+1\}^n$). We will write that $O_{x_0}$ is an orthant-neighbourhood of $(x_0,\epsilon)$\par

The orthant-neighbourhoods fit naturally with the use of Taylor expansion formula. Let us recall it relatively to an orthant-neigbourhood $(x_0,\epsilon)$ for a matrix polynomial $A(x)$ whose entries are not necessarily homogeneous polynomials :
$$A(x_0+\epsilon X)=\sum_{\alpha}\frac{(\epsilon X)^\alpha}{\alpha !} A^{(\alpha)}(x_0)$$
where we use the standart multi-index symbol for products.
Under a condition of domination by the constant term, we will see how to derive some certificate on orthant-neighbourhoods.\par

\begin{defn}
Let $A(x)$ be a matrix polynomial. Denote by $\Gamma_{x_0}$ the set of all multi-indexes $\beta$ which are minimal (for the lexicographic ordering) such that $A^{(\beta)}(x_0)\not =0$.\par
We say that $A(x)$ satisfies the domination condition at $x_0$ if for all multi-index $\alpha$ there is some multi-index $\beta\in\Gamma_{x_0}$ and a non negative real number $r_{\alpha,\beta}$ such that  $\beta\leq \alpha$ and $r_{\alpha,\beta} A^{(\beta)}(x_0)\pm A^{(\alpha)}(x_0)$ is psd. 
\end{defn}
Note that $\Gamma_{x_0}=\{0\}$ when $A(x_0)\not =0$ which simplifies the domination condition. Note also that if $\beta\in\Gamma_{x_0}$ then necessarily all its coordinates are even integers.\par Finally, mention that to check positivity via the Taylor formula, it would be enough to have the domination condition for any multi-index $\alpha$ whose at least one component is odd.

\air
A typical use of this domination condition appears in the following situation : 

\begin{prop}
Let $A(x)$ be a psd matrix polynomial on a neighbourhood of $x_0$. 
Then, $A(x)$ satisfies the  domination condition at $x_0$ if and only if it admits a semi-certificate at $x_0$ where the $U_{i,j}(x)$'s appearing in (\ref{semi-certif})  are constant matrices.
 \end{prop}
\begin{proof}
Assume that $A(x)$ satisfies the domination condition at $x_0$. Let $V$ be an orthant-neighbourhood of  $(x_0,\epsilon)$. 
We have
$$\begin{array}{ccl}
A(x_0+\epsilon X)&=&\sum_{\beta\in\Gamma_{x_0}}\sum_{\alpha\in\Delta_\beta\cup\{\beta\}}\frac{(\epsilon X)^\alpha}{\alpha !} A^{(\alpha)}(x_0)\\
\end{array}$$
where $\Delta_\beta$ is a subset of all multi-index $\alpha$ such that $\alpha> \beta$ (we have to be carefull that one multi-index $\alpha$ may be greater than several elements of  $\Gamma_{x_0}$).

$$\begin{array}{ccl}
A(x_0+\epsilon X)&=&\sum_{\beta\in\Gamma_{x_0}}
\left(\left(\epsilon^\beta\frac{X^\beta}{\beta !} A^{(\beta)}(x_0)\right)\left(1-\sum_{\Delta_\beta}\epsilon^{\beta}\frac{\beta !}{\alpha !}X^{\alpha-\beta} r_{\alpha,\beta}\right)\right.\\

&&\left. +\sum_{\Delta_\beta}\frac{X^{\alpha}}{\alpha !} \left(\epsilon^\alpha A^{(\alpha)}(x_0)+r_{\alpha,\beta} A^{(\beta)}(x_0)\right)\right)
\end{array}$$

Since $\epsilon^\beta=1$, we obtain a semi-certificate on the orthant-neighbourhood $V$.

\air
Conversely, assume that 
$$A(x)=\sum_if_i(x)V_i$$
where each $V_i$ is a constant psd matrix and each $f_i$ is non negative on a neighbourhood of $x_0$. We write the Taylor expansion of $f_i$ at $x_0$ :

$$f_i(x_0+X)V_i=\left(\sum_\alpha \frac{X^\alpha}{\alpha !}f_i^{(\alpha)}(x_0)\right)\times V_i$$
 
For the matrix polynomial $f_i(x)V_i$, consider the set $\Gamma_{x_0}$. If $\beta\in\Gamma_{x_0}$, then $f_i^{(\beta)}(x_0)>0$ and for any $\alpha>\beta$, we clearly have the existence of a positive real number $r_{\alpha,\beta}$ such that $r_{\alpha,\beta}f_i^{(\beta)}(x_0)\pm f_i^{(\alpha)}(x_0)\geq 0$. This is the domination condition.
\end{proof}

\remark{Since positive definite matrices obviously satisfy the domination condition, we may recover a version of Theorem \ref{main_th} with orthant-neighbourhoods. But, maybe (highly heuristic !) it will produce a lot more pieces for the covering, since (again roughly speaking) we need $2^{n-1}$ orthant-neighourhoods to recover a usual neighbourhood. } 
 
\subsection{Semi-certificates relative to a semi-algebraic subset}
One may naturally want to extend the framework of semi-certificates relatively to a basic closed semi-algebraic subset $S$. The problem is that the result given in this section does not take into account the equations describing $S$. We mainly use the underlying semi-algebraic set rather than the preordering or the quadratic module generated by the equations of $S$ as it is desired for a relative Positivestellensatz. In fact, this section concerns more the study of {\it local} semi-certificates rather than {\it relative}'s ones.  

\begin{thm}\label{relative_th} Let $A(x)\in \R[x]^{m\times m}$ be a homogeneous symmetric matrix polynomial. Assume that $A(x)$ is positive definite for all $x\in S$, where $S$ is a closed semi-algebraically subset of $\R^n$ defined by homogeneous polynomials. Then, there is a finite semi-algebraic covering of $S=\cup_{i=1}^rS_i$, some forms $p_{i,j}(x)\in\R[x]$, some homogeneous matrix polynomials $A_{i,j}(x)\in \R[x]^{m\times m}$ such that,  
$$\forall x\in S_i,\;A(x)=\sum_{j=1}^{r_i}p_{i,j}(x)A^T_{i,j}(x)A_{i,j}(x)$$
with the condition that $p_{i,j}(x)\geq 0$ for all $x\in S_i$.
\end{thm} 
\begin{proof}
We may perform the same proof as in Theorem \ref{main_th}.
\end{proof}

As an example of a case where $A(x)$ is not definite, we consider a general matrix polynomial of degree at most $2$ in a single (non homogeneous) variable $x$. 
\begin{example}
Assume that $A(x)$ is psd in the neighbourhood of $0^+$. \par
We assume moreover that $A(0)$ is not positive definite otherwise we are done by \ref{relative_th} or by a domination argument. Up to a base change, we may assume that $A(0)=\left(\begin{array}{cc}1&0\\0&0\end{array}\right)$. Then, let us write
$$A(x)=\left(\begin{array}{cc}1&0\\0&0\end{array}\right)+x\left(\begin{array}{cc}a_1&b_1\\b_1&c_1\end{array}\right)+x^2\left(\begin{array}{cc}a_2&b_2\\b_2&c_2\end{array}\right)=\left(\begin{array}{cc}1+a_1x+a_2x^2&b_1x+b_2x^2\\b_1x+b_2x^2&c_1x+c_2x^2\end{array}\right)$$
Since $A(x)$ is psd at $0^+$, we have $c_1\geq 0$. 

\begin{enumerate}
\item[*] If $c_1>0$, then the domination is satisfied and we get the following certificate
$$A(x)=(1-\mu x-\alpha x^2)\left(\begin{array}{cc}1&0\\0&0\end{array}\right)+x\left(\begin{array}{cc}a_1+\mu&b_1\\b_1&c_1\end{array}\right)+x^2\left(\begin{array}{cc}a_2+\alpha&b_2\\b_2&c_2+\beta c_1\end{array}\right)$$
where $\alpha,\beta,\mu$ are positive real numbers chosen such that the constant matrices of the identity are psd. Note that $(1-\mu x-\alpha x^2)$ is obviously positive on a neighbourhood of $0^+$.
\item[*]If $c_1=0$, then $A(x)$ is psd on $0^+$ if $c_2-b_1^2\geq 0$. If $c_2-b_1^2>0$, then we have the certificate
$$A(x)=(1-\alpha\epsilon x)\left(\begin{array}{cc}1&0\\0&0\end{array}\right)+\left(\begin{array}{cc}1-\epsilon&b_1x\\b_1x&\frac{(b_1x)^2}{1-\epsilon}\end{array}\right)+x^2\left(\begin{array}{cc}a_2+\alpha\epsilon&b_2\\b_2&c_2- \frac{b_1^2}{1-\epsilon}\end{array}\right)$$
where $\epsilon$ and $\alpha$ are positive  real numbers such that the last constant  matrix in the identify is positive definite.
\item[*] Now, consider the case when $c_1=0$ and $c_2-b_1^2=0$. By positivity of $A(x)$ we have $a_1c_2-2b_1b_2\geq 0$. \par
The case when $b_1=0$ is trivial since we must have $b_2=0$ and the certificate follows. If $b_1\not =0$, then we get the condition  $a_1b_1-2b_2\geq 0$. \par
First assume that  $a_1b_1-2b_2> 0$. We may write 
$$\begin{array}{ccl}
A(x)&=&\left(x\left(a_1-\frac{2b_2}{b_1}\right)+x^2\left(a_2-\left(\frac{b_2}{b_1}\right)^2\right)\right)\left(\begin{array}{cc}1&0\\0&0\end{array}\right)
\\
&&+\left(\begin{array}{cc}\left(1+\frac{b_2}{b_1}x\right)^2&(b_1x)\left(1+\frac{b_2}{b_1}x\right)\\(b_1x)\left(1+\frac{b_2}{b_1}x\right)&(b_1x)^2\end{array}\right)
\end{array}$$
\item[*] In the remaining case when  $c_1=0$, $c_2-b_1^2=0$, $b_1\not =0$, $a_1b_1-2b_2=0$, the positivity of $A(x)$ at $0^+$ says that $a_2b_1^2-b_2^2\geq 0$. The desired certificate follows from the identity given in the previous case. 
\end{enumerate}
\end{example}

This inspection of the most elementary situation leads to conjecture that any matrix polynomial in a single variable admits a local certificate of positivity.\par 
We give a proof of this fact, although we do not give explicit formulas depending on the entries as in the previous worked example. 

\begin{thm}
Let $M\in \R[x]^{n\times n}$ be a symmetric matrix polynomial whose entries are polynomials in a single variable $x$. Assume that $M$ is psd on a neighbourhood of $0^+$.\par  Then, $M$ admits a semi-certificate of positivity at $0^+$.
\end{thm}
\begin{proof}
If $M(0)$ is invertible, then we are done by \ref{main_th}. Hence, from now on, we assume that $\det(M(0))=0$.\par
Let us consider the Smith normal form of $M$ :
$$M=EDF$$
where $E$ and $F$ are invertible in $\R[x]^{n\times n}$ and $D=\diag(d_1,\ldots, d_n)$ is diagonal in $\R[x]^{n\times n}$ with diagonal entries $(d_1,\ldots, d_n)$. By changing $M$ to $(F^T)^{-1}MF^{-1}$, we may assume that $M=ED$. Moreover, if $d_i\equiv 0$ then we may argue by induction on the size of $M$, hence we will assume in the following that $d_i\geq 0$ at $0^+$ and vanishes only at $0$. \par
The matrix $M$ has the form
$$M=\left(\begin{array}{cccc}
d_1e_{1,1}&d_2e_{1,2}&\ldots&d_ne_{1,n}\\
d_2e_{1,2}&d_2e_{2,2}&\ldots&d_ne_{2,n}\\
\vdots&&&\vdots \\
d_ne_{1,n}&d_ne_{2,n}&\ldots&d_ne_{n,n}\\
\end{array}\right)$$
Hence the matrix $E$ has the form
$$E=\left(\begin{array}{ccccc}
e_{1,1}&e_{1,2}&\ldots&&e_{1,n}\\
\frac{d_2}{d_1}e_{1,2}&e_{2,2}&\ldots&&e_{2,n}\\
\vdots&\ddots&\ddots&&\vdots \\
\vdots&&\ddots&\ddots&\vdots \\
\frac{d_n}{d_1}e_{1,n}&\frac{d_n}{d_2}e_{2,n}&\ldots&\frac{d_n}{d_{n-1}}e_{2,n}&e_{n,n}\\
\end{array}\right)$$

First, let us introduce some notations.\par
Define by induction the following sequence of integers : \par Let $k_0=1$ and set $k_i$ to be the first integer $j>k_{i-1}$ such that $\frac{d_j}{d_{j-1}}(0)=0$. It yields an increasing sequence of integers : $$1=k_0<k_1<\ldots<k_{r}\leq n=k_{r+1}-1.$$
We divide the matrix $M=(m_{k,l})_{1\leq k,l\leq n}$ into block matrices : $M=(M_{i,j})_{0\leq i,j\leq r}$, where we set $M_{i,j}=(m^{(i,j)}_{k,l})_{1\leq k\leq k_{i+1}-k_{i},1\leq l\leq k_{j+1}-k_j}$ with $m^{(i,j)}_{k,l}=m_{k+k_i-1,l+k_j-1}$. Likewise, we divide $E=(E_{i,j})_{0\leq i,j\leq r}$ into similar block matrices.\par
Let us define the block matrix $M_i=(R_{k,l})_{0\leq k,l\leq r}$ such that $R_{k,l}=M_{k,l}$ when ($1\leq k \leq i$ and $l=i$) or ($1\leq l \leq i$ and $k=i$) and $R_{k,l}=0$ otherwise. Namely $$M_i=\left(\begin{array}{ccccccc}
0&\ldots&0&M_{1,i}&0&\ldots&0\\
\vdots&&\vdots&\vdots&\vdots&&\vdots\\
0&\ldots&0&M_{i-1,i}&\vdots&&\vdots\\
M_{i,1}&\ldots&M_{i,i-1}&M_{i,i}&\vdots&&\vdots\\
0&\ldots&\ldots&\ldots&0&\ldots&0\\
\vdots&&&&\vdots&&\vdots\\
0&\ldots&\ldots&\ldots&0&\ldots&0\\
\end{array}\right).$$
Let us write now

$$M=d_{k_0} M_{0} +d_{k_1}M_{1}+\ldots+d_{k_{r}}M_{r}.$$

Observe then by construction that $E(0)$ is block upper triangular. Since $E(0)$ is invertible we get that $E_{i,i}(0)$ is invertible for each $i=0,\ldots, r$.\par
Moreover, we have 


\begin{equation}\label{detinv}\det(M_{i,i})={\det}(E_{i,i})\times\prod_{j=k_{i}}^{k_{i+1}-1}d_{j}\end{equation}
By definition of the $k_i$'s, it implies that $\det(\frac{1}{d_{k_i}}M_{i,i})$ does not vanish at $0$ . Hence we may derive a kind of domination condition for the matrix $M$ at $0^+$. Namely, let us write 

$$\begin{array}{ccl}
M&=&d_{k_0} M_0 -d_{k_1}\times\mu_{1}\times \Id_{k_0,k_1}-\ldots-d_{k_r}\times\mu_{r}\times\Id_{k_0,k_{1}}\\
&&\\
& &+d_{k_1}(M_1+\mu_{1}\Id_{k_0,k_1})
-d_{k_2}\times\mu_{2}\times\Id_{k_1,k_2}-\ldots-d_{k_r}\times\mu_{r}\times\Id_{k_1,k_2}\\
&&\\
&&\vdots\\
&&\\
&&+d_{k_r}(M_r+\mu_{r}\Id_{k_0,k_r})\\
&&\\
&=&d_{k_0} N_0 +d_{k_1}N_1+\ldots+d_{k_r}{N}_r
\end{array}$$

where the $\mu_{i}$'s are positive real numbers and $\Id_{p,q}$ is the diagonal matrix in $\R^{n\times n}$ whose $i$-th diagonal entry is $+1$ if $p\leq i\leq q-1$ and $0$ otherwise. Because of (\ref{detinv}), we know that we can choose the $\mu_i$'s such that the matrix polynomials $N_i$ are psd on a neighbourhood of $0^+$.\par
Then, using the same argument as in Theorem \ref{main_th}, we complete the construction of a semi-certificate.
\end{proof}

The result is no more true with more than one variable :

\begin{prop}\label{counter_dim2}
Let $$M(x,y)=\left(
\begin{array}{cc}
1+x^2-y^2&-x\\
-x&y^2
\end{array}\right).$$ 
Then, $M(x,y)$  is psd for $x^2\geq y^2\geq 1$ and $x^2\leq y^2\leq 1$ althought it does not admit a certificate in a neighbourhood of $(1^+, 1^+)$.
\end{prop}     
\begin{proof}
First note that $\det(M(x,y))=(y^2-1)(x^2-y^2)$ to conclude to the positivity domain of $M(x,y)$. \par For convenience, let us consider the following change of variables $y=1+Y$ and $x=Y+H$. We also perform a base change to ``simplify" the expression of $M(0,0)=\left(
\begin{array}{cc}
1&-1\\
-1&1
\end{array}\right)$. Let $P=\left(
\begin{array}{cc}
1&-1\\
1&1
\end{array}\right)$ and $Q=\left(
\begin{array}{cc}
1&1\\
-1&1
\end{array}\right)$.
Then, $$P\times M(0,0)\times Q=\left(
\begin{array}{cc}
4&0\\
0&0
\end{array}\right)$$
and thus let us introduce the following matrix
$$\begin{array}{ccl}
N(X,Y)&=&P\times M(x,y)\times Q\\
&=&\left(
\begin{array}{cc}
4+4H+4Y+H^2+Y^2+2HY&2H-2Y+H^2-Y^2+2HY\\
2H-2Y+H^2-Y^2+2HY&H^2+Y^2+2HY
\end{array}\right)
\end{array}$$
The matrix $N(X,Y)$ is psd at an orthant-neighbourhood $(0^+,0^+)$ with respect to the variables $(H,Y)$. Let us assume that it has a semi-certificate of positivity. Namely :
\begin{equation}
N(X,Y)=\sum_i F_i\left(
\begin{array}{cc}
\epsilon_i^2&\epsilon_i\\
\epsilon_i&1
\end{array}\right)+
\sum_j \left(
\begin{array}{cc}
U_j^2&U_jV_j\\
U_jV_j&V_j^2
\end{array}\right)+G\left(
\begin{array}{cc}
1&0\\
0&0
\end{array}\right)
\end{equation}
where all $F_i$'s and $G$ are psd on an orthant-neighbourhood $(0^+, 0^+)$.
The degrees of the $(2,2)$-entries shows that each $f_i$ has only monomials of degree $2$ and each $V_j$ has monomials of degree $1$. Hence :
$$\left\{\begin{array}{l}
F_i=a_iH^2+b_iY^2+c_iHY\\
U_j=\gamma_j+\mu_jH+\nu_jY\\
V_j=\alpha_jH+\beta_jY\\
G={a'}H^2+{b'}Y^2+{c'}YH+{d'}H+{e'}Y+{f'}\\
\end{array}\right.$$

Identifying the $(2,2)$-entries, we get
\begin{equation}\label{entry22}
\left\{
\begin{split}
&\sum_ia_i+\sum_j\alpha_j^2=1, \\
& \sum_ib_i+\sum_j\beta_j^2=1\\
& \sum_ic_i+2\sum_j\alpha_j\beta_j=2\\
 \end{split}\right.
\end{equation}

Identifying the coefficients in $H$ and $Y$ of the $(1,2)$-entries and the constant coefficients of the $(1,1)$-entries yields :

\begin{equation}\label{before_CS}
\left\{
\begin{split}
& \sum_j\alpha_j\gamma_j=2\\
& \sum_j\beta_j\gamma_j=-2\\
& \sum_j\gamma^2_j +{f'}=4\\
 \end{split}\right.
\end{equation}
 
Let us introduce $\bar \alpha=(\alpha_j)$, $\bar \beta=(\beta_j)$ and
$\bar \gamma=(\gamma_j)$, vectors of the standart euclidean space (whose dimension is equal to the number of indexes $j$). \par
After $(\ref{before_CS})$ and $(\ref{entry22})$, we get 
$$\bar\gamma^2\leq 4,\quad \bar\alpha^2\leq 1\quad {\rm and}\quad \bar\alpha\cdot\bar\gamma=2.$$

By the Cauchy- Schwartz case of equality, we get $\bar\gamma^2=4$ and  $\bar\alpha^2=1$. Thus, ${f'}=0$ and $\sum_ia_i=0$ which means that for all $i$, $a_i=0$. 
Moreover, the vectors $\bar\alpha$ and $\bar\gamma$ must be colinear and hence $\bar\alpha=\frac{1}{2}\bar\gamma$. \par
Likewise $\beta=-\frac{1}{2}\gamma$ and $\sum_i c_i=4$. Moreover, for all $i$ we have  $b_i=0$. And hence we may assume that there is only one index $i$ and we may write 
\begin{equation}
N(X,Y)=4HY\left(
\begin{array}{cc}
\epsilon^2&\epsilon\\
\epsilon&1
\end{array}\right)+
\sum_j \left(
\begin{array}{cc}
U_j^2&U_jV_j\\
U_jV_j&V_j^2
\end{array}\right)+G\left(
\begin{array}{cc}
1&0\\
0&0
\end{array}\right).
\end{equation}

Identifying the $(1,2)$-entries, and setting $\bar\mu=(\mu_j)$ and $\bar\nu=(\nu_j)$ we have

\begin{equation}\label{entry12}
\left\{
\begin{split}
&\frac{1}{2}\bar\gamma\cdot\bar\mu =1\\
&\frac{1}{2}\bar\gamma\cdot\bar\nu =1\\
&\frac{\bar\gamma}{2}\cdot(\bar\mu-\bar\nu)+4\epsilon =2 \\
 \end{split}\right.
\end{equation}

We readily deduce that $\epsilon =\frac{1}{2}$. Identifying the $(1,1)$-entry, we have furthermore

\begin{equation}\label{entry11}
\left\{
\begin{split}
&2\,\bar\gamma\cdot\bar\mu+{d'}=4, \\
&2\,\bar\gamma\cdot\bar\nu+{e'}=4, \\
&2\,\bar\mu\cdot\bar\nu+{c'}+1=2, \\
&\bar\mu^2+{a'}=1, \\
&\bar\nu^2+{b'}=1, \\
 \end{split}\right.
\end{equation}

By $(\ref{entry12})$ and $(\ref{entry11})$, we immediately have ${d'}={e'}=0$. Moreover, we get also 
$$\bar\gamma\cdot\bar\mu=2, \quad \bar\mu^2\leq 1\quad {\rm and} \quad\bar\gamma^2=4.$$ By the Cauchy-Schwartz case of equality, we deduce that $a'=0$ and $\bar\mu=\frac{1}{2}\bar\gamma$. \par
Likewise ${b'}=0$ and $\bar\nu=\frac{1}{2}\bar\gamma$. \par
Then, $G=c'HY$ where necessarily ${c'}\geq 0$. 
To conclude, it suffices to note that $2\,\bar\mu\cdot\bar\nu+{c'}+1=2$ gives $2+{c'}+1=2$, a contradiction.
\end{proof}

Remember that any psd biquadratic form in dimension $2$ is a sum of squares, and hence admits a semi-certificate. This example shows that the {\it local} counterpart is no more true.

\air
In general we may mention the following result. Although completely elementary, it looks very much like the ones we can find in \cite{KS} or \cite{Pu} for hermitian squares or more classical sums of squares.

\begin{prop}
Let $A(x)\in\R[x]^{m\times m}$ be a symmetric matrix polynomial. Then, $A(x)$ is psd if and only if for all real $\epsilon>0$ the matrix polynomial  
$$A_\epsilon(x)=A(x)+\epsilon\left(\sum_{i=1}^nx_i^2\right)\Id_m$$ admits a semi-certificate.
\end{prop}

Of course the complexity of the semi-certificate may increaes as $\epsilon$ goes to zero.\par

\remark{
The complexity of a semi-certificate shall be measured by the size of the semi-algebraic partitions and also by the number of squares. Here the degrees of the polynomials appearing in a (homogeneous) semi-certificate are bounded by the degree of the given matrix polynomial.  
}

\remark{
By \cite{HS} with respect to the (compact) unit sphere $\{\sum_{i=1}^nx_i^2=1\}$, then we obtain for $M$ a Positivestellensatz which is no more homogeneous and where the degrees are no more bounded. 
}

\remark{
Consider a semi-certificate, and a compact semi-algebraic subset $S_i$ of the given partition. If we assume that the desciption of $S_i$ is archimedean, then we may deduce a ``true" algebraic certificate without denominators for the matrix polynomial on $S_i$, using for instance Schm\"udgen or Putinar certificates.}

\section{Around the Choi counterexample}
Let us consider the counterexample of a biqudratic psd non sum of squares given in \cite{C} : 
$$\left(\begin{array}{ccc}
x^2+2z^2&-xy&-xz\\
-xy&y^2+2x^2&-yz\\
-xz&-yz&z^2+2y^2\\
\end{array}\right)$$

Altough it is not a sum of squares, it admits a semi-certificate. Indeed, when 
$|x | \geq |z |$, it can be decomposed as 

$$\begin{array}{cl}
&
\left(\begin{array}{ccc}
x^2&-xy&-xz\\
-xy&y^2&yz\\
-xz&yz&z^2
\end{array}\right)+2\left(\begin{array}{ccc}
z^2&0&0\\
0&x^2&-yz\\
0&-yz&y^2\\
\end{array}\right)
\\
&\\
&=
(-x,y,z)^T(-x,y,z)
+2(0,-z,y)^T(0,-z,y)
\\
&\\
&+2z^2\left(\begin{array}{ccc}
1&0&0\\
0&0&0\\
0&0&0\\
\end{array}\right)
+2(x^2-z^2)\left(\begin{array}{ccc}
0&0&0\\
0&1&0\\
0&0&0\\
\end{array}\right)
\\
\end{array}$$
And by symmetry, we deduce an analogeous certificate when $|z | \geq |x |$.
\air

Let us consider a variation around this example ; in the following proposition, the Choi counterexample is just $M_1$ :

\begin{prop}\label{choi-variation}
Let $$M_\lambda=\left(\begin{array}{ccc}
x^2+(\lambda+1) z^2&-xy&-xz\\
-xy&y^2+(\lambda+1) x^2&-yz\\
-xz&-yz&z^2+(\lambda+1) y^2\\
\end{array}\right)$$
where $\lambda\in\R$.\par
Then, $M_\lambda$ is never a sum of squares and it is psd if and only if $\lambda\in[0,+\infty[$. Moreover, $\det(M_0)$ is a psd non sos polynomial and $M_\lambda$ admits a semi-certificate of positivity for any $\lambda\in]0,+\infty[$.
\end{prop}
\begin{proof}
The exact argument as in \cite{C} works to show that any $M_\lambda$ is not a sum of squares. We may reproduce it for convenience. We translate our ramewok into the language of biforms, setting
$$f_{M_\lambda}=(x^2+(\lambda+1)z^2)S^2-2xyST-2xzSU+(y^2+(\lambda+1) x^2)T^2-2yzTU+(z^2+(\lambda+1) y^2)U^2$$
Assume that $f_{M_\lambda}=\sum f_i^2$ where $f_i$ is a bilinear form. Since there is no monomials $x^2U^2$, $y^2S^2$, $z^2T^2$ in $f_{M_\lambda}$, there is no such monomials in each $f_i^2$, and hence no monomials $xU$, $yS$ nor $zT$ in each $f_i$. We may write $f_i=g_i+h_i$ were $g_i$ depends only in the monomials $xS$, $yT$, $zU$ and $h_i$ depends only in the monomials $xT$, $yU$, $zS$.\par
Then, $f_{M_\lambda}=\sum(g_i+h_i)^2$ shows that $$\sum_ig_i^2=x^2S^2+y^2T^2+z^2U^2-2(xyST+yzTU+xzSU),$$ a contradiction since the right hand side of the equality takes negative values for $x=y=z=S=T=U=1$.

\air
The fact that $M_0$ is psd and not a sum of squares may also be shown directly by using the following consequence of the Cauchy-Binet formula.\par

\begin{lem} Let $M(x)$ be a symmetric matrix polynomial. If $M$ is a sum of squares, then its determinant is also a sum of squares. \footnote{The Cauchy-Binet formula shows even more : if $M$ is a sum of squares, then {\it all} its principal minors are sums of squares.}. 
\end{lem}
\begin{proof}
We first state the Cauchy-Binet formula. \par Given matrices $A\in\R^{m\times s}$ and $B\in\R^{s\times m}$, the Cauchy-Binet formula states that 
$$\det(AB)=\sum_S\det(A_S)\det(B_S)$$ 
where $S$ ranges over all the subsets of $\{1,\ldots s\}$ with $m$ elements, and $A_S$ (respectively $B_S$) denotes the matrix in $\R^{m\times m}$ whose columns are the columns of $A$ (respectively whose rows are the rows of $B$) with index from $S$.\par
If $M(x)$ is a sum of squares, then it can be written $M(x)=A^T(x)A(x)$ for some matrix polynomial $A(x)\in\R[x]^{s\times m}$. Then, 
$$\begin{array}{ccl}
\det(M(x))&=&\sum_S\det(A(x)^T)_S\det(A(x)_S)\\
&=&\sum_S\det(M_S(x)^T)\det(M_S)\\
&=&\sum_S(\det(M_S))^2.
\end{array}$$
\end{proof}
Note that $M_0$ has as a determinant :
$$\det(M_0)=x^4y^2+y^4z^2+z^4x^2-3x^2y^2z^2$$
a variation of the celebrated Motzkin polynomial which is psd and not a sum of squares !\par
Thus, the matrix polynomial $M_0$, which is clearly psd since all its principal minors are psd, cannot be a sum of squares since its determinant is not !\air
Thus, $M_0$ (and thus all $M_\lambda$ for $\lambda\geq 0$) are psd matrix polynomial. Moreover, $$\det(M_\lambda(1,1,1))=\lambda(\lambda+3)^2$$
which shows that $M_\lambda$ is not psd for small $\lambda<0$ and hence for all $\lambda<0$.\air
Let us study now the existence of semi-certificate of positivity.\par

For $\lambda>0$, the set of real singular points of $\det(M_\lambda)$ have projective coordinates is $[1:0:0]$, $[0:1:0]$ and $[0:0:1]$ (they correspond to the points where the non negative polynomial $\det(M_\lambda)$ vanishes). 
The identity
$$M_\lambda=
\left(\begin{array}{ccc}
x^2&-xy&-xz\\
-xy&y^2&yz\\
-xz&yz&z^2
\end{array}\right)+
 \left(\begin{array}{ccc}
\lambda z^2&0&0\\
0& \frac{4z^2}{\lambda}&-2yz\\
0&-2yz&\lambda y^2\\
\end{array}\right)
+
\frac{1}{\lambda}(\lambda x^2-4z^2)\left(\begin{array}{ccc}
0&0&0\\
0& 1&0\\
0&0&0\\
\end{array}\right)
$$
gives a semi-certificate at the neighbourhood of $[1:0:0]$. We may also proceed likewise at the neighbourhood of $[0:1:0]$ and $[0:0:1]$.\par Elsewhere the matrix $M_\lambda$ is positive definite so that we can argue as in the proof of Theorem \ref{main_th} and obtain a finite semi-algebraic covering, and hence a semi-certificate. \air

Whereas, the matrice $M_0$ has a lot more singular points, namely : $[1:0:0]$, $[0:1:0]$, $[0:0:1]$, $[1:1:1]$, $[1:1:-1]$, $[1:-1:1]$, $[1:-1:1]$.\par

So, the problem which remains is to find a certificate of positivity for $M_0$ at the points $[1:1:1]$, $[1:1:-1]$, $[1:-1:1]$, $[1:-1:1]$. It does not appear trivial. For instance, the domination condition is not satisfied here. Another natural idea would be to consider :
$$\begin{array}{ccl}
M_0&=&

\left(\begin{array}{ccc}
x^2+z^2&-xy&-xz\\
-xy&y^2+x^2&-yz\\
-xz&-yz&z^2+y^2\\
\end{array}\right)\\
&&\\
&=&
\left(\begin{array}{ccc}
0&0&0\\
0&y^2&-yz\\
0&-yz&y^2\\
\end{array}\right)+\left(\begin{array}{ccc}
y^2&-xy&0\\
-xy&x^2&0\\
0&0&0
\end{array}\right) \\
&&\\
&+&
\left(\begin{array}{ccc}
x^2+z^2-y^2&0&-xz\\
0&0&0\\
-xz&0&y^2
\end{array}\right)
\end{array}$$

But, setting $z=1$ in the last matrix, we get a matrix polynomial which is psd at an orthant-neighbourhood of $(1^+,1^+)$ with respect to the variables $(x,y)$, but which does not admit any semi-certificate as shown by Proposition \ref{counter_dim2}. \par Nevertheless, let us see in the following how it is possible to get a semi-certificate in some orthant-neighbourhood of $[1:1:1]$, and hence  a partial result. 
\air
Let $z=1$ and $x=1+X$, $y=1+Y$. Write the Taylor formula at $(1,1)$ :
$$M_0=A_0+XA_1+X^2A_2+YB_1+Y^2B_2+XYC_2$$

Consider the following matrices :
$$A_X=\frac{5}{12}A_0+XA_1+X^2A_2=
\left(\begin{array}{ccc}
\frac{5}{6}+2X+X^2&-\frac{5}{12}-X&-\frac{5}{12}-X\\
&&\\
-\frac{5}{12}-X&\frac{5}{6}+2X+X^2&-\frac{5}{12}\\
&&\\
-\frac{5}{12}-X&-\frac{5}{12}&\frac{5}{6}
\end{array}\right)$$

$$A_Y =\frac{5}{12}A_0+YB_1+Y^2B_2=
\left(\begin{array}{ccc}
\frac{5}{6}&-\frac{5}{12}-Y&-\frac{5}{12}\\
&&\\
-\frac{5}{12}&\frac{5}{6}+2Y+Y^2&-\frac{5}{12}-Y\\
&&\\
-\frac{5}{12}&-\frac{5}{12}-Y&\frac{5}{6}+2Y+Y^2
\end{array}\right)$$

$$A_{XY} =\frac{2}{12}A_0+XYC_2=
\left(\begin{array}{ccc}
\frac{1}{3}&-\frac{1}{6}-XY&-\frac{1}{6}\\
&&\\
-\frac{1}{6}-XY&\frac{1}{3}&-\frac{1}{6}\\
&&\\
-\frac{1}{6}&-\frac{1}{6}&\frac{1}{3}
\end{array}\right)$$

One may check that 
$$C_X=A_X-A_0\left(\frac{X}{4}-X^2\right)=\left(\begin{array}{ccc}
\frac{18 X^{2}+9 X+5}{6} & \frac{-12 X^{2}-9 X-5}{12} & \frac{-12 X^{2}-9 X-5}{12} \\
&&\\
\frac{-12 X^{2}-9 X-5}{12} & \frac{18 X^{2}+9 X+5}{6} & \frac{-12 X^{2}+3 X-5}{12} \\
&&\\
\frac{-12 X^{2}-9 X-5}{12} & \frac{-12 X^{2}+3 X-5}{12} & \frac{12 X^{2}-3 X+5}{6}
\end{array}\right)$$
is positive semi-definite on all $\R$. Since it is a psd matrix polynomial with respect to the single variable $X$, it is a sum of squares.

\par
Likewise, one may check that 
$$C_Y=A_Y-A_0\left(\frac{-Y}{4}-Y^2\right)=
\left(\begin{array}{ccc}
\frac{12 Y^{2}+3 Y+5}{6} & \frac{-12 Y^{2}-15 Y-5}{12} & \frac{-12 Y^{2}-3 Y-5}{12} \\
&&\\
\frac{-12 Y^{2}-15 Y-5}{12} & \frac{18 Y^{2}+15 Y+5}{6} & \frac{-12 Y^{2}-15 Y-5}{12} \\
&&\\
\frac{-12 Y^{2}-3 Y-5}{12} & \frac{-12 Y^{2}-15 Y-5}{12} & \frac{18 Y^{2}+15 Y+5}{6}
\end{array}\right)$$
is positive semi-definite on all $\R$. Since it is a psd matrix polynomial with respect to the single variable $Y$, it is a sum of squares.\par
Finally, write
$$A_{XY}=\left(\frac{2(1+6XY)}{12}A_0\right)-XY(A_0-C_2)$$
Where we check that $(A_0-C_2)$ is a psd constant matrix.\par
If we sum up all these informations, we obtain a desired certificate of positivity
$$M_0=A_X+A_Y+A_{XY}=A_0\left(\frac{X}{4}-X^2\right)+C_X+A_0\left(\frac{-Y}{4}-Y^2\right)+C_Y+A_{XY}$$
with respect to the open orthant-neighbourhood of $(0,0)$ defined by $X>0$ and $Y<0$. 
\air
Likewise, we may produce a local certificate with respect to the orthant defined by $X<0$ and $Y>0$). \par But, it seems less clear how to obtain a local certificate with respect to the orthants defined by $XY>0$.

\end{proof}

\section{Concluding remarks}
\subsection{Open questions about semi-certificates}
We have introduced the notion of piecewise semi-certificate of positivity for matrix polynomials. For the moment, the only general result is that all positive definite matrix polynomial admit such a certificate. A lot of things remain to be studied.\air
\begin{itemize}
\item
We shall better understand the set of all psd matrix polynomials which admit a semi-certifiacte. Beginning with biquadratic forms for instance ?
\item We shall developp some effective algorithm to produce the certificates. 
\par
\item About the complexity of certificates : how can we  bound the number of squares, and also the number of pieces of the semi-algebraic partition ? 
\par
\end{itemize}
On the other hand, one may also rise some questions about the familly of psd polynomials which have a quadratic determinantal representation. This is the object of the last subsection :

\subsection{Semi-definite quadratic determinatal representations}
In the spirit of what happens in Proposition \ref{choi-variation}, one may be interested in determining what psd polynomial can be written as the determinant of a psd quadratic matrix polynomial (whose entries are homogeneous polynomials of degree $2$).
Namely, 
\begin{question}\label{deter-quadrat}
Let $f(x)\in\R[x]$ be a form of degree $2d$ which is suppose to be non negative. Does-there exists a quadratic matrix polynomial $M(x)\in\R[x]^{d\times d}$ which is psd for all $x\in\R^n$ and such that
\begin{equation}\label{determ-quadratic}f(x)=\det(M(x)) ?\end{equation}
\end{question}

One motivation for such a question is that it appears as the quadratic analogeaous of very classical linear determinantal representations which have been studied for a very long time. The specifically real considerations being more recents and esentially du to Vinnikov (see for instance \cite{V} and all related papers). Let us recall that these  have a lot of applications for instance to Linear Matrix Inequalities, convex modelling, etc...\par
On the other hand, biquadratic forms are a quite common object found at various areas in ingenering applications. For instance, any determinantal representation as in (\ref{determ-quadratic}) of a psd non sum of squares polynomial provides, by Cauchy-Binet formula, another example of a psd biquadratic form which is not a sum of squares.\air
 
Of course, a simple count on the number of parameters shows Question \ref{deter-quadrat} has a negative answer in general.\par
We can be even more precise : if $f$ is a form of degree $4$ in $4$ variables, such a representation never exists when $f$ is psd non sum of squares. Indeed, if $f=\det(M(x))$ with $M(x)\in\R[x]^{2\times 2}$, then $M(x)$ is a psd biquadratic form which is a sum of squares (in dimension $2$), hence $\det(M)$ is a sum of squares of polynomials by Cauchy-Binet formula.\par
We may also look at the trivial case when $n=2$, which we may deshomogenize for simplicity. The answer to Question \ref{deter-quadrat} is clearly positive by an elementary argument. Indeed, let us write $$f(x)=\sum_{i=0}^{2d}a_ix^i=\prod_j(x-\alpha_j)^2\times\prod_{k}(x+\beta_k)^2+\gamma_k^2)$$
Then, the result follows obviously from the multiplicative property of the determinant and the fact that the $(x-\alpha_i)^2$'s and the $(x+\beta_k)^2+\gamma_k^2)$'s are psd quadratic polynomials.\par
 
We may even give a more algorithmic  construction (a polynomial time algorithm with respect to the size of the coefficients of the polynomial) using arrows matrices and following \cite{Fi} and \cite{Qz}. 
\air
Back to Question \ref{deter-quadrat}, and before expecting general results, one may wonder if for instance the celebrated Motzkin and Robinson polynomials \par
$$\left\{\begin{array}{l}
{\rm Mo}=z^6+x^2y^4+x^4y^2-3\,x^2y^2z^2\\
\\
{\rm Ro}=x^6+y^6+z^6-x^4y^2-x^4z^2-y^4x^4-y^4z^2-z^4x^2-z^4y^2+3\,x^2y^2z^2\\
\end{array}\right.$$
can be written as in (\ref{determ-quadratic}) ?
 \air

It is quite easy, using for instance the linear well-know determinantal representations
of cubics curves, to produce quadratic determinantal representations for these two polynomials, but unfortunately none such representation yields a {\it psd} quadratic matrix polynomial.

\end{document}